\documentclass[a4paper,10pt]{amsart}

\newfont{\cyr}{wncyr10 scaled 1100}

\usepackage[left=2.7cm,right=2.7cm,top=3.5cm,bottom=3cm]{geometry}
\usepackage{amsthm,amssymb,amsmath,amsfonts,mathrsfs,amscd, graphics}
\usepackage[latin1]{inputenc}
\usepackage[all]{xy}
\usepackage{latexsym}
\usepackage{longtable}
\usepackage{color}
\usepackage{tikz-cd}
\usepackage{hyperref}

\numberwithin{equation}{section}

\theoremstyle{plain}
\newtheorem{theorem}{Theorem}[section]
\newtheorem*{theorem*}{Theorem}
\newtheorem{corollary}[theorem]{Corollary}
\newtheorem{lemma}[theorem]{Lemma}
\newtheorem{proposition}[theorem]{Proposition}

\newtheorem{axiom}[theorem]{Axiom}
\newtheorem{thm}{Theorem}

\theoremstyle{definition}

\theoremstyle{remark}
\newtheorem{obswr}[theorem]{Observation}
\newtheorem{remarkwr}[theorem]{Remark}
\newtheorem{example}[theorem]{Example}

\newenvironment{remark}{\begin{remarkwr}\begin{upshape}}{\end{upshape}\end{remarkwr}}

\def\Gal{\mathrm{Gal}}

\def\GSp{\mathrm{GSp}}

\def\Sp{\mathrm{Sp}}

\def\supp{\mathrm{supp}}

\DeclareMathOperator{\Adm}{Adm}

\def\FF{\mathbb{F}}

\def\QQ{\mathbb{Q}}
\def\RR{\mathbb{R}}

\def\ZZ{\mathbb{Z}}

\include{thebibliography}

\begin{document}

\title[Deligne-Lusztig varieties and EKOR strata]
{Deligne-Lusztig varieties and basic EKOR strata}

\author{Haining Wang }

\begin{abstract}
Under the axiom of He and Rapoport for the stratifications of Shimura varieties, we explain a result of  G\"{o}rtz, He and Nie that the EKOR strata contained in the basic loci can be described as a disjoint union of Deligne-Lusztig varieties. In the special case of Siegel modular varieties, we  compare their descriptions to that of  G\"{o}rtz and Yu for the supersingular Kottwitz-Rapoport strata and to the descriptions of Harashita, Hoeve for the supersingular Ekedahl-Oort strata.
\end{abstract}

\address{\parbox{\linewidth} {Haining Wang,\\ Department of Mathematics,\\ McGill University,\\ 805 Sherbrooke St W,\\ Montreal, QC H3A 0B9, Canada.~ }}
\email{wanghaining1121@outlook.com}

\date{\today}

\maketitle

\tableofcontents

\section*{Introduction}
Let $\mathcal{A}_{g}$ be the moduli space of principally polarized abelian varieties over $\bar{\FF}_{p}$. There are two interesting stratifications of this moduli space. The first one is the \emph{Newton stratification} which stratify the space according to the Newton polygons of the abelian varieties. The other one is the \emph{Ekedahl-Oort stratification} which stratify the space according to the isomorphism classes of the $p$-torsion subgroups of the abelian varieties. One of the most interesting strata in the Newton stratification is the basic stratum or the supersingular stratum. It is a classical question how the Ekedahl-Oort stratification interacts with the Newton stratification. In particular, one can ask how the Ekedahl-Oort stratification meets the basic Newton stratum. This question was answered in \cite{Har10} by Harashita where he showed that certain unions of the supersingular Ekedahl-Oort strata are isomorphic to Deligne-Lusztig varieties. This result is refined by Hoeve in \cite{Hoe10} and he showed that each individual supersingular Ekedahl-Oort stratum is isomorphic to a fine Deligne-Lusztig variety. For general Shimura varieties that have good integral model and stratification theory, one can still formulate the same questions. Vollaard and Wedhorn \cite{Vo10, VW11} studied the case of unitary Shimura varieties of signature $(1, n-1)$ at an innert prime and answered this question. Note in this case, the basic Newton stratum is very special in the sense that it can be covered entirely by Ekedahl-Oort strata and again each Ekedahl-Oort stratum is a union of Deligne-Lusztig varieties. Those Shimura varieties that satisfy the similar special properties have been classified by the work of G\"{o}rtz, He and Nie, see \cite{GH15, GHN19}. They termed these special Shimura varieties as the fully Hodge-Newton decomposable Shimura varieties. In fact they not only treated the case of Ekedahl-Oort stratifications on Shimura varieties with good reductions  but also the general cases of Shimura varieties with parahoric level structures. Note in the case of Shimura varieties with parahoric level structures, one has to replace the Ekedahl-Oort strata by the so called EKOR (Ekedahl-Oort-Kottwitz-Rapoport) strata. This notion is introduced by He and Rapoport in \cite{HR17} where they axiomatized the theory of stratifications on general Shimura varieties with parahoric level structures. This notion interpolates the Ekedahl-Oort stratification and the so-called Kottwitz-Rapoport stratification. The EKOR strata have been studied extensively in the recent preprint \cite{SYZ19} in the cases of Hodge type and abelian type Shimura varieties. The Kottwitz-Rapoport stratification was first studied by Ng\^{o} and Genestier in the case of Siegel modular varieties with Iwahori level structure, \cite{GN02}. It can be considered as a stratification by singularities as the semisimple traces of Frobenius on the nearby cycle sheaf is constant on each stratum. In \cite{GY10, GY12}, G\"{o}rtz and Yu studied those Kottwitz-Rapoport strata contained in the basic Newton strata on the Siegel modular varieties with Iwahori level structures. They showed again that those Kottwitz-Rapoport strata can be described as a union of Deligne-Lusztig varieties. Finally, we remark that the appearances of Deligne-Lusztig varieties in the basic Newton strata are not only beautiful in its own right but are also important in arithmetic applications. Notably, they play an important role in the study of arithmetic fundamental lemma \cite{RTZ13, HLZ19} and they are also crucial in establishing geometric versions of Jacquet-Langlands correspondence \cite{HTX17, XZ17}.

The first aim of the present note is to present a result of  G\"{o}rtz, He and Nie that generalizes the aforementioned results of G\"{o}rtz and Yu as well as the results of Harashita and Hoeve. We formulate their results in the setting of Shimura varieties that satisfy the He-Rapoport axiom. Here we only state an informal version and the details can found in the proof of Theorem \ref{Main-1}. 
\begin{thm}[G\"{o}rtz-He-Nie]
Every EKOR stratum that is contained in the basic Newton stratum can be written as a disjoint union of classical Deligne-Lusztig varieties.
\end{thm}
This theorem is not stated anywhere in the literature but can be deduced from \cite{GH15, GHN19}. Here we give a quite self-contained presentation of this result.  We also make an effort to bring this result in the frame work of He-Rapoport by introducing an additional axiom incorporating the Langlands-Rapoport conjecture, see axiom \ref{uniformize}. We hope that our presentation will be useful to workers in this trade, especially those who are not so familiar with \cite{GH15, GHN19}. The salient feature of the above theorem is that once the group theoretic data giving rise to the Shimura varieties and their stratifications are at hand, the descriptions of the basic EKOR strata can be made explicit without reference to the Shimura varieties themselves. This is especially valuable when the Shimura variety does not admits a good moduli interpretation. We make the above descriptions of the EKOR strata explicit in the case of Siegel modular varieties both in the case of hyperspecial level and in the case of Iwahori level. In particular this allows us to compare the descriptions of the basic EKOR strata by G\"{o}rtz-He-Nie with those obtained by G\"{o}rtz-Yu  and by Harashita, Hoeve. 
\begin{thm}
The descriptions of the basic EKOR strata of G\"{o}rtz-He-Nie for Siegel modular varieties agree with 
\begin{enumerate}
\item the descriptions of G\"{o}rtz-Yu for the basic KR strata;
\item the decriptions of Hoeve for the basic Ekedahl-Oort strata.
\end{enumerate}
\end{thm}
The descriptions of the KR strata for Siegel modular varieties with Iwahori level will be discussed in \S 5.2 and the proof of $(1)$ of the theorem above can be found in Theorem \ref{GY-comp}.  The descriptions of the Ekedahl-Oort strata for the Siegel modular varieties with hyperspecial level will be discussed in \S 5.3 and $(2)$ is proved in Theorem \ref{Ho-comp}.

\subsection*{Acknowledgement} This work is completed when the author is a postdoctoral fellow at McGill university and he would like to thank Henri Darmon and Pengfei Guan for their generous support.

\section{Preliminaries}
\subsection{The Iwahori-Weyl group}\label{preli} Let $p$ be an odd prime. Let $F$ be a non-archimedean local field with valuation ring $\mathcal{O}_{F}$ and residue field $k$ containing $\FF_{p}$. We denote by $\breve{F}$ the completion of the maximal unramified extension of $F$ inside a fixed algebraic closure $\bar{F}$ with  $\breve{\mathcal{O}}_{F}$ the valuation ring and $\FF$ its residue field. Let $\sigma$ be the Frobenius automorphism acting on $\FF$ and we use the symbol for its lift on $\breve{F}$. We set $\Gamma=\Gal(\bar{F}/F)$ and $\Gamma_{0}=\Gal(\bar{F}/\breve{F})$ which we identify with the inertia group of $F$. 

Let $G$ be a connected reductive group over $F$. We write $\breve{G}$ for its base change to $\breve{F}$ which is quasi-split. We choose a maximal $\breve{F}$-split torus $S$ of $G$ and denote by $T$ its centralizer. Since $\breve{G}$ is quasi-split, $T$ is a maximal torus of $G$. Let $N_{T}$ be the normalizer of $T$ in $G$. Then we define the \emph{finite Weyl group} $W$  by $W=N_{T}(\breve{F})/T(\breve{F})$ with its natural action by $\sigma$. The \emph{Iwahori-Weyl  group} is defined similarly by 
\begin{equation*}
\tilde{W}=N_{T}(\breve{F})/T(\breve{F})_{1},
\end{equation*}
where $T(\breve{F})_{1}$ is the unique parahoric subgroup of ${T}(\breve{F})$. The torus $S$ defines an apartment $\mathfrak{A}$ in the Bruhat-Tits building of ${G}$ over $\breve{F}$.  Its underlying affine space is given $V=X_{*}(T)_{\Gamma_{0}}\otimes_{\ZZ} \RR$ and $\tilde{W}$ acts on it by affine transformations. The action $\sigma$ on $\breve{T}$ induces an action on $V$ and we choose a $\sigma$-invariant alcove $\mathfrak{a}$ which we will refer to as the base alcove. The stabilizer in $\breve{G}$ of the base alcove $\mathfrak{a}$ is the \emph{Iwahori subgroup} $\breve{I}$ in ${G}(\breve{F})$ corresponding to $\mathfrak{a}$. We denote by $V_{+}$ the closure of the dominant Weyl chamber in $V$.  We also fix a special vertex in the closure of $\mathfrak{a}$ which is invariant under the Frobenius of the unique quasi-split inner form of $G$. These choices give us a splitting of the Iwahori-Weyl group
\begin{equation*}
\tilde{W}=X_{*}(T)_{\Gamma_{0}}\rtimes W.
\end{equation*}
When there is no harm of confusion we will use the same symbol $w$ to denote an element in $\tilde{W}$ and its lift in $N_{T}(\breve{F})$. For an element $\lambda\in X_{*}(T)_{\Gamma_{0}}$, we will write $t^{\lambda}$ when we consider $\lambda$ as element of $\tilde{W}$. The stabilizer of $\mathfrak{a}$ in $\tilde{W}$ will be denoted by $\Omega$ and we have a decomposition $\tilde{W}=W_{a}\rtimes \Omega$ where $W_{a}$ is the \emph{affine Weyl group} of $G$. This is the same as the Iwahori-Weyl group of the simply connected cover of the derived group of $G$.  Thus we have
\begin{equation*} 
W_{a}=X_{*}(T_{sc})_{\Gamma_{0}}\rtimes W
\end{equation*}
where $T_{sc}$ is the simply connected cover of the image of $T$ in the derived group. Note that $X_{*}(T_{sc})$ is the coroot lattice and therefore the quotient $X_{*}(T)/X_{*}(T_{sc})$ is by definition the fundamental group $\pi_{1}(G)$ of $G$. We then arrive at the following relationship between the Iwahori-Weyl group and the affine Weyl group 
\begin{equation*}
\tilde{W}=W_{a}\rtimes \pi_{1}(G)_{\Gamma_{0}}.
\end{equation*}
And we can identify $\Omega$ with $\pi_{1}(G)_{\Gamma_{0}}$.  We also obtain in this way a map 
\begin{equation}
\kappa_{\tilde{W}}: \tilde{W}\rightarrow \pi_{1}(G)_{\Gamma}
\end{equation}
by further projecting the natural map $\tilde{W}\rightarrow \pi_{1}(G)_{\Gamma_{0}}$ to $\pi_{1}(G)_{\Gamma}$.
The group $W_{a}$ is Coxeter group and we denote by $\tilde{\mathbb{S}}$ the set of simple reflections in $W_{a}$. One has a Bruhat order on it and a well-defined length function on $W_{a}$. We extend them to $\tilde{W}$ in the usual way: if $w=w_{0}\tau\in W_{a}\rtimes \pi_{1}(G)_{\Gamma_{0}}$ with $w_{0}\in W_{a}$ and $\tau\in \pi_{1}(G)_{\Gamma_{0}}$, then we define $l(w)=l(w_{0})$ and we write $w_{1}\tau_{1}\leq w_{2}\tau_{2}$ if and only if $\tau_{1}=\tau_{2}$ and $w_{1}\leq w_{2}$. 

Let $K\subset \tilde{\mathbb{S}}$ be a subset of the set of simple reflections. We denote by $W_{K}$ be the subgroup of $\tilde{W}$ generated by the simple reflections in $K$ and denote by $^{K}\tilde{W}$ the set of minimal length representatives in the cosets of $W_{K}\backslash \tilde{W}$. 
\subsection{$\sigma$-conjugacy classes} We let $B(G)$ be the set of $\sigma$-conjugacy classes of $G(\breve{F})$. Then we have the \emph{Newton map} 
\begin{equation*}
\nu_{G}: B(G)\rightarrow ((X_{*}(T)_{\Gamma_{0}, \QQ})^{+})^{\langle\sigma\rangle}
\end{equation*}
where $X_{*}(T)^{+}_{\Gamma_{0}, \QQ}=X_{*}(T)_{\Gamma_{0}, \QQ}\cap V_{+}$.
And the \emph{Kottwitz map}
\begin{equation*}
\kappa_{G}: B(G)\rightarrow \pi_{1}(G)_{\Gamma}.
\end{equation*}
For a review of the definitions of these maps, we refer the reader to \cite{RR96} and \cite{RV14}, we only remark that $\nu_{G}$ should be considered as the group theoretic incarnation of the notion of Newton polygons and $\kappa_{G}$ encodes the end point of the Newton polygon. The joint map
\begin{equation*}
\nu_{G}\times \kappa_{G}: B(G)\rightarrow ((X_{*}(T)_{\Gamma_{0}, \QQ})^{+})^{\langle\sigma\rangle}\times\pi_{1}(G)_{\Gamma}.
\end{equation*}
is in fact injective. The set $B(G)$ is equipped with a partial order: we write $[b]\leq [b^{\prime}] $ for $[b], [b^{\prime}]\in B(G)$ if $\kappa_{G}([b])=\kappa_{G}([b^{\prime}])$ and $\nu_{G}([b])\leq\nu_{G}([b^{\prime}])$. Here $\nu_{G}([b])\leq\nu_{G}([b^{\prime}])$ if $\nu_{G}([b])-\nu_{G}([b^{\prime}])$ is non-negative $\QQ$-sum of positive coroot.  

Let $\{\mu\}$ be a conjugacy class of cocharacters over $\bar{F}$ of $G$. An element $[b]\in B(G)$ is called \emph{acceptable} for $\{\mu\}$ if $\nu([b])\leq \bar{\mu}$ where $\bar{\mu}=[\Gamma: \Gamma_{\mu}]^{-1}\sum_{\gamma\in \Gamma/\Gamma_{\mu}}\gamma(\mu)$ is the Galois average of $\mu$ which we define to be the unique dominant element in $\{\mu\}$ with $\Gamma_{\mu}$ its stabilizer in $\Gamma$. An acceptable element $[b]$ is said be \emph{neutral} if $\kappa([b])=\mu^{\natural}$ with $\mu^{\natural}$ the common image of $\{\mu\}$ in $\pi_{1}(G)_{\Gamma}$, see \cite[Lemma 3.1]{Rap05}. We define the set of \emph{neutral acceptable elements} with respect to $\{\mu\}$ in $B(G)$ by
\begin{equation}
B(G, \{\mu\})=\{[b]\in B(G): \kappa([b])=\mu^{\natural}, \nu([b])\leq \bar{\mu}\}. 
\end{equation}
This is in fact a finite set and inherits a partial order from that of $B(G)$ and the unique minimal class is called the \emph{basic class}.  We denote by 
\begin{equation}\label{basic-tau}
\tau=\tau_{\{\mu\}}
\end{equation}
the unique element of $\tilde{W}$ of length $0$ mapping to $\mu^{\natural}\in\pi_{1}(G)_{\Gamma_{0}}$ and its $\sigma$-conjugacy class is the basic class defined above. For $w\in \tilde{W}$, we consider $w\sigma$ as an element in $\tilde{W}\rtimes \langle\sigma\rangle$. There exist $n$ such that $(w\sigma)^{n}=t^{\lambda}$ for an element $\lambda\in X_{*}(T)_{\Gamma_{0}}$. Then the \emph{Newton vector} $\bar{\nu}_{w}$ of $w$ is defined to be the unique dominant element in the $W$-orbit of  $\lambda/n\in X_{*}(T)_{\Gamma_{0}, \QQ}$. An element $w\in \tilde{W}$ is called $\sigma$-straight if $l((w\sigma)^{m})=ml(w)$ for all $m\in \mathbb{N}$ where we extend the length function $l(\cdot)$ to $\tilde{W}\rtimes \langle\sigma\rangle$ by requiring that $l(\sigma)=0$. We denote by $B(\tilde{W})_{\sigma}$ the set of $\sigma$-conjugacy classes in $\tilde{W}$ and we say a $\sigma$-conjugacy class straight if it contains a $\sigma$-straight element. We denote by $B(\tilde{W})_{\sigma-\mathrm{str}}$ the set of straight $\sigma$-conjugacy classes of $\tilde{W}$. This set is closely related to the Kottwitz set $B(G)$, in fact we have the following theorem 
\begin{theorem}[{\cite[Theorem3.3]{He14}}]\label{W-B}
The map $\Psi: B(\tilde{W})_{\sigma-\mathrm{str}}\rightarrow B(G)$ induced by the inclusion of $N_{T}(\breve{F})\subset G(\breve{F})$ is a bijection. 
\end{theorem}

We finally remark that the set $B(G, \{\mu\})$ should be thought of as the collection of Newton polygons that satisfy the Mazur's inequality and is the index set of the Newton stratification of Shimura varieties.

\subsection{The admissible set} Consider again a conjugacy class $\{\mu\}$ of $G$ over $\bar{F}$. Recall we have fixed a dominant representative $\mu\in \{\mu\}$ and we use the same notation for its image in $X_{*}(T)_{\Gamma_{0}}$. The \emph{admissible set } for $\{\mu\}$ is defined by
\begin{equation*}
\mathrm{Adm(\{\mu\})}=\{x\in\tilde{W}: x\leq t^{w(\mu)}\text{ for some finite Weyl group element } w\}. 
\end{equation*}
From the definition we immediately obtain $\Adm(\{\mu\})\subset W_{a}\tau$. Suppose $K\subset \tilde{\mathbb{S}}$ is a $\sigma$-invariant subset. We set 
\begin{equation*}
\begin{split}
&\mathrm{Adm}^{K}(\{\mu\})=W_{K}\mathrm{Adm}(\{\mu\})W_{K}\subset \tilde{W};\\
&\mathrm{Adm}(\{\mu\})_{K}=W_{K}\backslash\mathrm{Adm}^{K}(\{\mu\})/W_{K}\subset W_{K}\backslash \tilde{W}/W_{K}.\\
\end{split}
\end{equation*}
Let $\Adm(\{\mu\})_{\mathrm{str}}$ be set of the $\sigma$-straight elements in $\Adm(\{\mu\})$ and let $B(\tilde{W}, \{\mu\})_{\mathrm{str}}$ be the image of it in $B(\tilde{W})$. The following theorem refines the previous Theorem \ref{W-B}. 

\begin{theorem}
The map $\Psi$ restricts to a bijection between $B(\tilde{W}, \{\mu\})_{\mathrm{str}}$ and $B(G, \{\mu\})$. 
\end{theorem}

\section{EKOR stratifications of Shimura varieties}
\subsection{Shimura varieties} Let $({\bf{G}}, \{h\})$ be a Shimura datum and let ${\bf{K}}={\bf{K}}^{p}{\bf{K}}_{p}$ be an open compact subgroup of ${\bf{G}}(\mathbb{A}_{f})$ with ${\bf{K}}^{p}$ sufficiently small and ${\bf{K}}_{p}$ be a standard parahoric subgroup $G(\QQ_{p})$ with $G={\bf{G}}_{\QQ_{p}}$ containing a fixed Iwahori subrgoup $I$ of $G(\QQ_{p})$. Let $\mathrm{Sh}_{{\bf{K}}}=\mathrm{Sh}({\bf{G}}, \{h\})_{{\bf{K}}}$ be the corresponding Shimura variety defined over a reflex field ${\bf{E}}$. We denote by $E$ the completion of ${\bf E}$ at a place ${\bf p}$ above $p$ and $\mathcal{O}_{E}$ its valuation ring. Let $k_{E}$ be the residue field of $E$. We assume that there exits a suitable integral model $\mathcal{S}h_{{\bf{K}}}$ of $\mathrm{Sh}_{{\bf{K}}}$ over $\mathcal{O}_{E}$. Let $Sh_{K}$ be the special fiber of $\mathcal{S}h_{{\bf{K}}}$ over $k_{E}$.  The Shimura datum $({\bf{G}}, \{h\})$ gives a conjugacy class of cocharacters $\{\mu\}$ of $G$ defined over $E$. Let $\breve{K}$ be the parahoric subgroup of ${G}(\breve{\QQ}_{p})$ corresponding to ${\bf{K}}_{p}$. Let $\mathcal{G}_{K}$ be the Bruhat-Tits group scheme over $\ZZ_{p}$ correponding to ${\bf{K}}_{p}$ and we set $\mathcal{G}_{K, k_{E}}=\mathcal{G}_{K}\otimes_{\ZZ_{p}} k_{E}$. We write $G(\breve{\QQ}_{p})/\breve{K}_{\sigma}$ for the $\sigma$-conjugacy class of $G(\breve{\QQ}_{p})$ by $\breve{K}$.  In particular, we have $B(G)=G(\breve{\QQ}_{p})/G(\breve{\QQ}_{p})_{\sigma}$. 

\subsection{Stratifications for Shimura varieties}
We assume that $Sh_{K}$ satisfies the axioms of He-Rapoport in \cite{HR17}. We will not recall these axioms one by one, instead we record the following commutative diagram whose existence follows from the axioms of He and Rapoport.

\begin{center}
 \begin{tikzcd}[row sep=tiny]
  & & \breve{K} \backslash G(\breve{\QQ}_{p})/\breve{K}\\
  Sh_{K} \arrow[urr, bend left, "\lambda_{K}" ] \arrow[drr, bend right, "\delta_{K}"]  \arrow[r, "\gamma_{K}"] & G(\breve{\QQ}_{p})/\breve{K}_{\sigma} \arrow [ur, "l_{K}"] \arrow[dr, "d_{K}"]\\
  & &  B(G) 
\end{tikzcd}
\end{center}

\begin{theorem}[\cite{HR17}] \label{im}
We have 
\begin{enumerate}
\item $\mathrm{Im}(\lambda_{K})\subset \mathrm{Adm}(\{\mu\})_{K}$;
\item $\mathrm{Im}(\delta_{K})\subset B(G, \{\mu\})$;
\item $\mathrm{Im}(\gamma_{K})=\bigsqcup_{w\in \mathrm{Adm}(\{\mu\})_{K}}\breve{K}w\breve{K}/K_{\sigma}$.
\end{enumerate}
\end{theorem}
We remark that $(1)$ is proved in \cite[Proposition 3.13 (i)]{HR17}. The map $\lambda_{K}$ is related to the map from $Sh_{K}$ to the stack $[M^{\mathrm{loc}}_{K}/\mathcal{G}_{K, k_{E}}]$ where $M^{\mathrm{loc}}_{K}$ is the special fiber of the local model of $Sh_{K}$.  In many cases, one can embed this special fiber in an affine flag variety which can be decomposed into Schubert cells which are indexed by $\breve{K}\backslash G(\breve{\QQ}_{p})/\breve{K}$, \cite{Gor01, Gor03}. It is expected that the set $\mathrm{Adm}(\{\mu\})_{K}$ is precisely the index set for those cells that could "see" the Shimura variety $Sh_{K}$. Therefore we define the \emph{Kottwitz-Rapoport stratum} associated to $w\in \Adm(\{\mu\})_{K}$ by 
\begin{equation}
KR_{K, w}=\lambda^{-1}_{K}(w).
\end{equation}
This is a locally closed subvariety of $Sh_{K}$.

Next $(2)$ is proved in  \cite[Proposition 3.13 (ii)]{HR17}. The map $\delta_{K}$ is related to the \emph{Newton stratification} and $B(G, \{\mu\})$ is the natural index set of the Newton stratification. We define the \emph{Newton stratum} associated to an element $[b]\in B(G, \{\mu\})$ by
\begin{equation}
S_{K,[b]}=\delta^{-1}_{K}([b]).
\end{equation}
This is a locally closed subvariety of $Sh_{K}$.

Finally $(3)$ is proved in \cite[Corollary 4.2]{HR17}. Let $\breve{K}_{1}$ be the pro-unipotent radical of $\breve{K}$. Then we have the following inclusions
\begin{equation*}
\breve{K}_{\sigma}\subset \breve{K}_{\sigma}(\breve{K}_{1}\times \breve{K}_{1})\subset \breve{K}\times \breve{K}. 
\end{equation*}
We consider the following composition of two maps
\begin{equation*}
\nu_{K}: Sh_{K}\rightarrow G(\breve{\QQ}_{p})/\breve{K}_{\sigma}\rightarrow  G(\breve{\QQ}_{p})/\breve{K}_{\sigma}(\breve{K}_{1}\times \breve{K}_{1})
\end{equation*}
where the first one is the map $\gamma_{K}$ and the second map is the natural projection map. We will need the following decomposition theorem of the group $G(\breve{\QQ}_{p})$. 
\begin{theorem}[{\cite[Theorem 6.1]{HR17}}]\label{decomp}
Let $K$ be a parahoric subgroup. Then
\begin{equation*}
G(\breve{\QQ}_{p})=\bigsqcup_{x\in ^{K}\tilde{W}}\breve{K}_{\sigma}(\breve{K}_{1}x\breve{K}_{1})=\bigsqcup_{x\in ^{K}\tilde{W}}\breve{K}_{\sigma}(\breve{I}x\breve{I}).
\end{equation*}
\end{theorem}
Combing the previous theorems, we obtain the following lemma.
\begin{lemma}
\begin{equation*}
\mathrm{Im}(\nu_{K})\subset \bigsqcup_{w\in\Adm(\{\mu\})^{K}\cap {^{K}\tilde{W}}} \breve{K}_{\sigma}(\breve{K}_{1}w\breve{K}_{1})/\breve{K}_{\sigma}(\breve{K}_{1}\times \breve{K}_{1}).
\end{equation*}
\end{lemma}
\begin{proof}
The lemma follows from Theorem \ref{im} and Theorem \ref{decomp}.
\end{proof}
 We in fact the following fact \cite[Theorem 6.1]{He16}
 \begin{equation}
  \Adm(\{\mu\})\cap {^{K}\tilde{W}}=\Adm(\{\mu\})^{K}\cap {^{K}\tilde{W}}. 
 \end{equation}
 For an element $w\in \Adm(\{\mu\})\cap {^{K}\tilde{W}}$, we define the \emph{EKOR stratum} of $Sh_{K}$ by the following subset
\begin{equation}\label{EKOR-stra}
EKOR_{K, w}= \nu^{-1}_{K}(\breve{K}_{\sigma}(\breve{I} w\breve{I})/\breve{K}_{\sigma})\subset Sh_{K}. 
\end{equation}
Notice that if ${K}={I}$ is the Iwahori subgroup, then $\Adm(\{\mu\})\cap {^{K}\tilde{W}}=\Adm(\{\mu\})$ and 
\begin{equation*}
EKOR_{I, w}=KR_{I, w}.
\end{equation*} 

On the other hand, suppose $G$ is unramified over $\QQ_{p}$ and we let $K=G(\ZZ_{p})$. Then consider the following set \cite{Vie14}
\begin{equation}
\mathcal{T}=\{(w, \mu)\in W\times X_{*}(T)^{+}: w\in {^{\mu}W}\}.
\end{equation} 
Here $^{\mu}W$ means the minimal length representatives in the coset space $W_{\mu}\backslash W$ with $W_{\mu}$ the parabolic subgroup determined by $\mu$. 

\begin{example}\label{Siegel}
Suppose $G=\GSp_{2g}$ and $\mu=(1^{(g)}, 0^{(g)})$. Then $W_{\mu}=W_{J}$ with $J=\{s_{1}, \cdots, s_{g}\}-\{s_{g}\}$. One can even identify $W_{\mu}$ with $S_{g}$ the symmetric group with $g$ letters. Then $^{\mu}W$ is exactly the index set for the \emph{Ekedhal-Oort stratification} on $\mathcal{A}_{g}$ the moduli space of principal polarized abelian schemes of dimension $g$. The strata classify the isomorphism classes of the $p$-torsion subgroups of the abelian schemes on $\mathcal{A}_{g}$. 
\end{example}
The main result of \cite{Vie14} gives us a bijection between $\mathcal{T}$ and the set of $\sigma$-$\breve{K}$-conjugacy classes of $\breve{K}_{1}\backslash G(\breve{\QQ}_{p})/\breve{K}_{1}$ and author denotes this bijection by
\begin{equation}
tr: G(\breve{\QQ}_{p})/\breve{K}_{\sigma}(\breve{K}_{1}\times \breve{K}_{1})\rightarrow \mathcal{T}.
\end{equation}
Now consider the map $\nu_{K}$ in this setting, as we have known the image is contained in  
\begin{equation}\label{EKOR}
\bigsqcup_{w\in\Adm(\{\mu\})^{K}\cap {^{K}\tilde{W}}} \breve{K}_{\sigma}(\breve{K}_{1}w\breve{K}_{1})/\breve{K}_{\sigma}(\breve{K}_{1}\times \breve{K}_{1})\end{equation}
and by the definition of $\Adm(\{\mu\})^{K}$ the restriction of $tr$ to
$$\bigsqcup_{w\in\Adm(\{\mu\})^{K}\cap {^{K}\tilde{W}}} \breve{K}_{\sigma}(\breve{K}_{1}w\breve{K}_{1})/\breve{K}_{\sigma}(\breve{K}_{1}\times \breve{K}_{1})$$ is mapped to $^{\mu}W$. Therefore in this case we call the EKOR stratum indexed by $w$ simply the \emph{Ekedahl-Oort stratum}  (\emph{EO} stratum) indexed by $w$ and we will write
\begin{equation}
EO_{w}=EKOR_{K,w}.
\end{equation}
From the above discussions, we see that the EKOR stratification indeed interpolates the KR stratification and the Ekedahl-Oort stratification. 

\section{Affine Deligne-Lusztig varieties}
\subsection{Deligne-Lusztig varieties} Before we discuss affine Deligne-Lusztig varieties, we first introduce classical Deligne-Lusztig varieties. In this subsection, we will abuse the notations and denote by $G$ a connected reductive algebraic group over $\FF$ which is an algebraic closure of a finite field $\FF_{q}$ with $q$ a power of $p$. Let $B$ be a Borel subgroup of $G$ over $\FF_{q}$ with a Levi decomposition $B=TU$ which is defined over $\FF_{q}$. Let $W=N_{T}(\FF)/T(\FF)$ be the associated Weyl group with an action of $\sigma$. We denote by $\mathbb{S}$ the set of simple reflections generating $W$. Let $w\in W$ and we denote by $\mathrm{supp}(w)$ the support of $w$ which is the set of simple reflections that occur in some reduced expression of $w$.  An element $w\in W$ is called a $\sigma$-\emph{Coxeter} element if $w$ is a product of simple reflections, each of which belongs to a unique $\sigma$-orbit of $\mathbb{S}$. 

Recall we have the Bruhat decomposition 
\begin{equation*}
G=\bigsqcup_{w\in W}BwB. 
\end{equation*}
Then the \emph{Deligne-Lusztig variety} associated to $w\in W$ is defined to be 
\begin{equation*}
Y(w)=\{gB\in G/B: g^{-1}\sigma(g)\in BwB/B\}.
\end{equation*}
Here $\sigma$ is the Frobenius acting on $G$ and $X(w)$ is considered as a locally closed subvariety of the flag variety $G/B$. 

Let $J\subset \mathbb{S}$. Let $W_{J}$ be the corresponding parabolic subgroup of $W$ and $P_{J}$ be corresponding parabolic subgroup of $G$. Recall that ${^{J}W}$ is the set of minimal length representatives of $W_{J}\backslash W$. For any $w\in {^{J}W}$, we set 
\begin{equation*}
Y_{J}(w)= \{gP_{J}: g^{-1}\sigma(g)\in {P_{J}}_{\sigma}(BwB)\}.
\end{equation*}
This is called the \emph{fine Deligne-Lusztig variety}. 

\subsection{Affine Deligne-Lusztig varities} Now we move back to the set-up in \S\ref{preli}. In particular, we assume for simplicity that $G$ is a connected reductive group over $\QQ_{p}$. Let $b\in G(\breve{\QQ}_{p})$ be a representative of the $\sigma$-conjugacy class $[b]$. Let $K\subset \tilde{\mathbb{S}}$ and $\breve{K}\subset G(\breve{\QQ})$ be the corresponding parahoric subgroup. Then the \emph{affine Deligne Lusztig variety} is given by the set
\begin{equation*}
X_{w}(b)_{K}=\{g\breve{K}: g^{-1}b\sigma(g)\in \breve{K}w\breve{K}\}.
\end{equation*}
This the $\FF$-points of a locally closed subscheme of the $p$-adic partial affine flag variety $G(\breve{\QQ}_{p})/\breve{K}$, \cite{BS17, Zhu17}.  Note that the group of $\sigma$-centralizer 
\begin{equation}
J_{b}=\{g\in G(\breve{\QQ}_{p}): g^{-1}b\sigma(g)=b\}
\end{equation}
acts naturally on $X_{w}(b)_{K}$. 

It is natural to consider the following variant of the above set which we will call the \emph{generalized affine Deligne-Lusztig variety}
\begin{equation*}
X(\mu, b)_{K}=\{g\breve{K}: g^{-1}b\sigma(g)\in \bigcup_{w\in \Adm^{K}(\{\mu\})}\breve{K}w\breve{K} \}.
\end{equation*}
It is obvious we have a natural decomposition
\begin{equation}\label{decomp-adlv}
X(\mu, b)_{K}=\bigcup_{w\in \Adm^{K}(\{\mu\})} X_{w}(b)_{K}. 
\end{equation} 
We also need the notion of a \emph{fine Deligne-Lusztig variety}. Let $w\in {^{K}\tilde{W}}$, we define 
\begin{equation*}
X_{K,w}(b)=\{g\breve{K}: g^{-1}b\sigma(g)\in \breve{K}_{\sigma}(\breve{I}w\breve{I})\}.
\end{equation*}
The decomposition in \eqref{decomp-adlv} can be made finer as in \cite[3.4]{GH15} and we have
\begin{equation*}
X(\mu, b)_{K}=\bigsqcup_{w\in {^{K}\tilde{W}}\cap\Adm^{K}(\{\mu\})} X_{K, w}(b).
\end{equation*}
In light of the EKOR stratifications of the Shimura varieties introduced in \eqref{EKOR-stra} and relation between the generalized affine Deligne-Lusztig variety and the Rapoport-Zink space that we will discuss later. We will refer to this decomposition of the generalized Deligne-Lusztig variety as the EKOR stratification of $X(\mu, b)_{K}$. The affine Deligne-Lusztig variety shows up in the stratification of Shimura varieties in the following explicit way, \cite[6.2]{GHN19}. We will restrict the map $\gamma_{K}$ to the Newton stratum $S_{K,[b]}=\delta^{-1}_{K}([b])$ and obtain
\begin{equation}\label{Newton-ADLV}
\gamma_{K}(S_{k,[b]})=d^{-1}_{K}([b])\cap l^{-1}_{K}(\Adm^{K}(\{\mu\}))=J_{b}\backslash X(\mu, b)_{K}. 
\end{equation}

\section{The basic EKOR strata} 
\subsection{Basic Newton Stratum} Recall the Langlands-Rapoport conjecture \cite{LR87, Rap05} asserts the following decomposition of the point of $Sh_{K}$ valued in $\FF$
\begin{equation}
\bigsqcup_{\phi}I_{\phi}(\QQ)\backslash X_{p}(\phi)\times X^{p}(\phi)/K^{p} \xrightarrow{\cong} Sh_{K}(\FF).
\end{equation}
We will not recall the general form for this conjecture. But if the Shimura variety is a moduli space of abelian varieties, then $\phi$ is supposed to run through the isogeny classes on the moduli space. The space $X_{p}(\phi)$ is supposed to correspond to the part coming from the $p$-power isogenies and $X^{p}(\phi)$ should correspond to the prime to $p$-part of the isogenies. In this case $I_{\phi}$ is the algebraic group over $\QQ$ corresponding to the automorphism group of an abelian variety in the isogeny class. Furthermore $X_{p}(\phi)$ should be given by a suitable affine Deligne-Lusztig variety. Recall that we have a fixed element $\tau$ in the basic class in $B(G, \mu)$. We consider now the associated basic Newton stratum $S_{K, [\tau]}$. In light of the above discussion on the Langlands-Rapoport conjecture, we postulate the following axiom in this note. 
\begin{axiom}\label{uniformize}
The basic Newton stratum is uniformized by the affine Deligne-Lusztig variety under a uniformization map
\begin{equation}
\iota_{K}: X(\mu, \tau)_{K}\rightarrow S_{K, [\tau]} 
\end{equation}
which induces an isomorphism of the form
\begin{equation}\label{uniform}
I_{\tau}(\QQ)\backslash X(\mu, \tau)_{K}\times X^{p}(\tau)/K^{p}\xrightarrow{\cong} S_{K, [\tau]}
\end{equation}
where $X^{p}(\tau)$ is a $G(\mathbb{A}^{p}_{f})$-torsor and $I_{\tau}$ is an algebraic group over $\QQ$ that acts on $X(\mu, \tau)_{K}$ through a map $I_{\tau}(\QQ)\rightarrow J_{\tau}(\QQ_{p})$. In addition, we require that the the composite of the map 
\begin{equation*}
\gamma_{K}\circ\iota_{K}: X(\mu, \tau)_{K}\rightarrow S_{K,[\tau]}\rightarrow J_{\tau}\backslash X(\mu, \tau)_{K}
\end{equation*}
is the natural map. 
\end{axiom}

In the case for the Hodge type Shimura varieties, the uniformization map is constructed in \cite{Kis17} in the unramified case and in \cite{Zhou19} for the parahoric case. Another way to postulate this is to require that there is an isomorphism of perfect schemes
\begin{equation*}
\mathcal{M}(G, \mu, \tau)^{per}_{K}\cong X(\mu, \tau)_{K} 
\end{equation*}
where $\mathcal{M}(G, \mu, b)^{per}_{K}$ is the perfection of the local Shimura variety (Rapoport-Zink space) \cite{RV14} that could be used to uniformize the basic stratum in an analogues way as in \eqref{uniform}, see \cite[(5.4)]{RV14}.

\subsection{The basic EKOR stratum} From here on, we will be concerned with the \emph{basic EKOR strata}. In other words, we will be concerned with those $w\in \Adm(\{\mu\})\cap {^{K}\tilde{W}}$ such that
\begin{equation*}
EKOR_{K, w}\subset S_{K,[\tau]}.
\end{equation*}
For $w\in W_{a}$, we denote by $\supp(w)$ the support of $w$, the set of $s_{i}\in \tilde{\mathbb{S}}$ that appears in some reduced expression of $w$. We also define  
\begin{equation*}
\supp_{\sigma}(w\tau)=\bigcup_{n\in \ZZ} (\tau\sigma)^{n}(\supp(w)).
\end{equation*}
We have the following numerical criterion to characterize the basic EKOR strata.
\begin{proposition}[{\cite[Proposition 5.6]{GHN19}}]\label{basic-EKOR}
Let $w\in \Adm(\{\mu\})\cap {^{K}\tilde{W}}$, then $EKOR_{K, w}\subset S_{K,[\tau]}$ if and only if $W_{\supp_{\sigma}(w)}$ is finite.
\end{proposition}
\begin{proof}
Note that $EKOR_{K, w}\subset S_{K,[\tau]}$ if and only if $\breve{K}_{\sigma}(\breve{I}w\breve{I})\subset [\tau]$.  By \cite[Proposition 5.6]{GHN19}, this happens if and only if $W_{\supp_{\sigma}(w)}$ is finite and $\kappa_{\tilde{W}}(w)=\kappa_{G}(\tau)$. Since $w\in \Adm(\{\mu\})$, it follows that $w\leq t^{w_{0}(\mu)}$ for some finite Weyl group element $w_{0}$ and $\kappa_{\tilde{W}}(w)=\kappa_{\tilde{W}}(t^{w_{0}(\mu)})=\kappa_{\tilde{W}}(w_{0}t^{\mu}w^{-1}_{0})=\kappa_{G}(\tau)$. 
\end{proof}

Let $w\in \tilde{W}$ and $K\subset \tilde{\mathbb{S}}$. We write $\mathrm{Ad}(w)\sigma(K)=K$ if  for any $s_{k}\in K$ there exists $s_{k^{\prime}}\in K$ with $w\sigma(s_{k})w^{-1}=s_{k^{\prime}}$. The set of $\{K^{\prime}\subset K: \mathrm{Ad}(w)\sigma(K^{\prime})\subset K^{\prime}\}$ contains a unique maximal element \cite[3.1]{GH15} which we denote by 
\begin{equation}\label{setI}
I(K, w, \sigma)= \mathrm{max} \{K^{\prime}\subset K: \mathrm{Ad}(w)\sigma(K^{\prime})\subset K^{\prime}\}.
\end{equation}

\begin{proposition}[{\cite[Proposition 5.7]{GHN19}}] 
Let $K$ be a $\sigma$-invariant subset of $\tilde{\mathbb{S}}$ and $w\in {^{K}\tilde{W}}$. If $W_{\supp_{\sigma}(w)}$ is finite, then $W_{\supp_{\sigma}(w)\cup I(K, w, \sigma)}$ is also finite. 
\end{proposition}

Let $w\in  \Adm(\{\mu\})\cap {^{K}\tilde{W}}$ be an element such that $\supp_{\sigma}(w)$ is finite. Then $EKOR_{K, w}$ and the basic Newton stratum $S_{K, [\tau]}$ fits in the following diagram under the uniformizatiom map introduced in \eqref{uniform}
\begin{equation*}
\begin{tikzcd}
X_{K,w}(\tau)   \arrow[r]\arrow[d] &EKOR_{K, w} \arrow[r] \arrow[d] & J_{\tau}\backslash X_{K,w}(\tau)   \arrow[d]\\
 X(\mu, \tau)_{K} \arrow[r, "\iota_{K}"]  &S_{K, [\tau]} \arrow[r,"\gamma_{K}"] & J_{\tau}\backslash X(\mu, \tau)_{K}.
\end{tikzcd}
\end{equation*} 

\begin{corollary}\label{EKOR-DL}
Let $w\in  \Adm(\{\mu\})\cap {^{K}\tilde{W}}$. The basic EKOR stratum $EKOR_{K, w}$ can be uniformized by the fine Deligne-Lusztig variety $X_{K, w}(\tau)$ in the sense of Axiom \ref{uniformize} that is
\begin{equation*}
I_{\tau}(\QQ)\backslash X_{K, w}(\tau)\times X^{p}(\tau)/K^{p}\xrightarrow{\cong} EKOR_{K, w}.
\end{equation*}
\end{corollary}
Using the above corollary, we can restrict our attention to the fine affine Deligne-Lusztig variety $X_{K, w}(\tau)$ in order to study the basic EKOR stratum $EKOR_{K, w}$. We recall the following results  of G\"{o}rtz and He \cite[Proposition 4.1.1, Theorem 4.1.2]{GH15} concerning the decompositions of fine affine Deligne-Lusztig varieties into fine Deligne-Lusztig varieties.

\begin{theorem}[{\cite[Theorem 4.1.1]{GH15}}]\label{fine-coarse}
For $J\subset \tilde{\mathbb{S}}$ and $w\in {^{J}\tilde{W}}$, 
\begin{equation*}
X_{J, w}(b)=\{g\breve{K}_{I(J, w, \sigma)}: g^{-1}b\sigma(g)\in \breve{K}_{I(J, w, \sigma)}w\breve{K}_{\sigma(I(J, w, \sigma))}\}.
\end{equation*}
\end{theorem}

\begin{proposition}[{\cite[Proposition 4.1.2]{GH15}}]\label{affine-classical}
Let $J\subset \tilde{\mathbb{S}}$ and $w\in {^{J}\tilde{W}^{\sigma(J)}}\cap W_{a}\tau$ such that  $Ad(w)\sigma(J)\subset J$. If $W_{\supp_{\sigma}(w)\cup J}$ is finite, then 
\begin{equation*}
\{g \breve{K}_{J}: g^{-1}\tau\sigma(g)\in \breve{K}_{J}w\breve{K}_{\sigma(J)}\}=\bigsqcup_{i\in J_{\tau}/J_{\tau}\cap \breve{K}_{\supp_{\sigma}(w)\cup J}}iY_{J}(w)
\end{equation*}
where $Y_{J}(w)=\{g\breve{K}_{J}\in \breve{K}_{\supp_{\sigma}(w)\cup J}/\breve{K}_{J}: g^{-1}\tau\sigma(g)\in \breve{K}_{J}w\breve{K}_{\sigma(J)}\}$ is a classical Deligne-Lusztig variety.
\end{proposition}

\begin{theorem}[G\"{o}rtz-He-Nie]\label{Main-1}
Every basic EKOR stratum on  $Sh_{K}$ can be written as a disjoint union of classical Deligne-Lusztig varieties.
\end{theorem}
\begin{proof}
By Theorem \ref{fine-coarse}, 
\begin{equation*}
X_{K, w}(\tau)=\{g\breve{K}_{I(K, w, \sigma)}: g^{-1}\tau\sigma(g)\in \breve{K}_{I(K, w, \sigma)}w\breve{K}_{\sigma(I(K, w, \sigma))}\}
\end{equation*}
where $I(K,w, \sigma)$ is defined as in \eqref{setI} and in particular $Ad(w)\sigma(I(K,w, \sigma))\subset I(K,w, \sigma)$. It follows from Proposition \ref{affine-classical} that we have 
\begin{equation*}
\begin{split}
X_{K, w}(\tau)= & \{g\breve{K}_{I(K, w, \sigma)}: g^{-1}\tau\sigma(g)\in \breve{K}_{I(K, w, \sigma)}w\breve{K}_{\sigma(I(K, w, \sigma))}\} \\
 & = \bigsqcup_{i\in J_{\tau}/J_{\tau}\cap \breve{K}_{\supp_{\sigma}(w)\cup I(K, w, \sigma)}}iY_{I(K, w, \sigma)}(w). \\
\end{split}
\end{equation*}
\end{proof}

\begin{remark}
In the case $w$ is a $\sigma$-Coxeter element in the finite group $W_{\supp_{\sigma}(w)}$, the Deligne-Lusztig variety $Y_{I(K, w, \sigma)}(w)$ is isomorphic to the following fine Deligne-Lusztig variety by \cite[Corollary 4.6.2]{GH15}
\begin{equation*}
Y_{I(K, w, \sigma)}(w)=\{g \breve{K}_{\supp_{\sigma}(w)\cap K} \in \breve{K}_{\supp_{\sigma}(w)}/\breve{K}_{\supp_{\sigma}(w)\cap K}: g^{-1}\tau\sigma(g)\in {\breve{K}_{\supp_{\sigma}(w)\cap K \sigma}}\breve{I}w\breve{I} \}
\end{equation*}
and $Y_{I(K, w, \sigma)}(w)$ is in turn isomorphic to the classical Deligne-Lusztig vareity
\begin{equation*}
\{g\breve{I}\in \breve{K}_{\supp_{\sigma}(w)}/\breve{I}: g^{-1}\tau\sigma(g)\in \breve{I}w\breve{I}\}.
\end{equation*}
\end{remark}

\section{Basic EKOR strata on Siegel modular varieties}
\subsection{Siegel modular varieties} We first review the Siegel moduli spaces of abelian varieties with parahoric level structures. The algebraic group associated to this Shimura variety is $G=\GSp_{2g}$ the group of symplectic similitudes. We consider the dominant miniscule coweight $\mu=(1^{(g)}, 0^{(g)})$ as in Example \ref{Siegel}. The finite Weyl group $W_{g}$ is generated by the simple reflections given by
\begin{equation*}
s_{g}=(g, g+1)\text{ and } s_{i}=(i,i+1)(2g-i, 2g+1-i)\text{ for }i=1, \cdots, g-1. 
\end{equation*}
The cocharacter group of $G$ with respect to the diagonal torus can be identified with 
\begin{equation*}
X_{*}(T)=\{(x_{1}, \cdots, x_{2g})\in \ZZ^{2g}:x_{1}+x_{2g}=x_{2}+x_{2g-1}=\cdots=x_{g}+x_{g+1}\}.
\end{equation*}
The set of simple reflections $\tilde{\mathbb{S}}$ in the affine Weyl group is $\{s_{0}, s_{1}, \cdots, s_{g}\}$ with
\begin{equation*}
s_{0}=((-1, 0, \cdots, 0, 1), (1, 2g))\in X_{*}(T)\rtimes W.
\end{equation*}
We will use the notation $I:=\{0, 1, 2, \cdots, g\}$ as the index set as well as the notation for the Iwahori subgroup.  Consider 
\begin{equation*}
\tau=((0^{(g)}, 1^{(g)}), (1,g+1)(2, g+2)\cdots(g, 2g))
\end{equation*}
the unique element of length $0$ in $\Adm({\{\mu\}})$. Note the class $[\tau]$ gives the unique basic class in $B(G, \mu)$.

Let $J\subset I$ given by $\{j_{0}<j_{1}<\cdots< j_{r}\}$. The moduli space $\mathcal{A}_{J}$ over $\FF_{p}$ with parahoric level structure of type $J$ parametrizes chains of $g$-dimensional polarized abelian varieties
\begin{equation*}
(A_{j_{0}}\xrightarrow{\alpha} A_{j_{1}}\xrightarrow{\alpha}\cdots \xrightarrow{\alpha} A_{j_{r}}, \eta)
\end{equation*}
where 
\begin{enumerate}
\item $A_{j_{i}}$ is a $g$-dimensional abelian variety and $\lambda_{j_{i}}$ a polarization of degree $p^{2(g-j_{i})}$ whose kernel is contained in $A_{j_{i}}[p]$.
\item $\alpha$ is an isogeny that pulls back $\lambda_{j_{i+1}}$ to $\lambda_{j_{i}}$.
\item $\eta$ is a symplectic level $N$ structure with $N\geq 3$ coprime to $p$ with a fixed primitive $N$-th root of unity.
\end{enumerate}
Given a chain $(A_{j_{i}}, \eta)_{j_{i}\in J}$, we can extend the index set $J$ to $J\cup\{2g-j_{i}: j_{i}\in J\}$ by duality:
\begin{equation*}
A_{j_{0}}\rightarrow A_{j_{1}}\rightarrow\cdots \rightarrow A_{j_{r}} \rightarrow A_{2g-j_{r}}:=A^{\vee}_{j_{r}}\rightarrow \cdots \rightarrow A_{2g-j_{0}}:=A^{\vee}_{j_{0}}.
\end{equation*}

When $J=I$, this recovers the Siegel moduli space with Iwahori level structure $\mathcal{A}_{I}$. When $J=\{0\}$, then $\mathcal{A}_{J}$ is the classical moduli space $\mathcal{A}_{g}$ of principally polarized $g$-dimensional abelian varieties. For simplicity, we will concentrate on these two cases in the following.

\subsection{Basic KR stratum} On $\mathcal{A}_{I}$, the relative position of $\mathrm{H}_{1}^{\mathrm{dR}}(A_{i})$ and $\omega_{A^{\vee}_{i}}$ gives a well-defined element $w\in\Adm(\{\mu\})$, see \cite[3.1]{GY12}. The resulting map
\begin{equation*}
\lambda_{I}: \mathcal{A}_{I}\rightarrow \Adm(\{\mu\})
\end{equation*}
gives the KR stratification for $\mathcal{A}_{I}$. For $w\in \Adm(\{\mu\})$, we will denote the corresponding KR stratum by $\mathcal{A}_{I, w}$. The local Dynkin diagram of type $\tilde{C}_{g}$ is given by 

\begin{displaymath}
 \xymatrix{\underset{0}\bullet \ar@{=>}[r]& \underset{1}\bullet \ar@{-}[r] &\underset{2}\bullet &\underset{3}\bullet\ar@{-}[l]\ar@{-}[r] &\cdots\ar@{-}[r]&\underset{g-1}\bullet\ar@{<=}[r]&\underset{g}\bullet}.
\end{displaymath}
The operator $\tau\sigma$ acts on the diagram and exchanges the nodes $i$ and $g-i$ for $i=0, \cdots, g$. Recall that the KR stratum $\mathcal{A}_{I, w}$ is contained in the basic Newton stratum if and only if the Weyl group $W_{\supp_{\sigma}(w)}$ is finite by Proposition \ref{basic-EKOR}. To simplify the notation, in this section we denote by $W_{\{c, g-c\}}$ the Weyl group generated by the simple reflections in $\tilde{\mathbb{S}}-\{c, g-c\}$.

\begin{proposition}\label{basic-fin}
Let $w\in \Adm(\{\mu\})$. The Weyl group $W_{\supp_{\sigma}(w)}$ is finite if and only if $w\in W_{\{c, g-c\}}\tau$ for some $c\leq g/2$.
\end{proposition}
\begin{proof}
One direction is clear that is if $w\in W_{\{c, g-c\}}\tau$, then $W_{\supp_{\sigma}(w)}$ is obviously finite. On the other hand, if $W_{\supp_{\sigma}(w)}$ is finite, then there is some $s_{c}$ that is not in $\supp_{\sigma}(w)$. It follows then that $s_{g-c}\not\in \supp(w\tau^{-1})$. Therefore $w\in  W_{\{c, g-c\}}\tau$.
\end{proof}

In \cite{GY10}, the authors call a KR stratum $\mathcal{A}_{I, w}$ \emph{superspecial} if there is an $0\leq c\leq [g/2]$ such that for all $(A_{j}, \eta)_{j\in I}\in \mathcal{A}_{I,w}$, the abelian varieites $A_{c}$ and $A_{g-c}$ are  superspecial and that the isogeny $A_{c}\rightarrow A^{\vee}_{g-c}$ is the Frobenius map. By \cite[Proposition 4.4]{GY10}, the KR stratum $\mathcal{A}_{I,w}$ associated to $w$ is superspecial if and only if $w\in W_{\{c, g-c\}}\tau$ for some $0\leq c\leq [g/2]$. Therefore we have recovered the following theorem of G\"{o}rtz and Yu.
\begin{theorem}[{\cite[Theorem 1.4]{GY10}}]
All basic KR strata are superspecial.
\end{theorem}

Next we explain how to describe the basic KR stratum $\mathcal{A}_{I, w}$ in terms of a disjoint union of Deligne-Lusztig varieties as in Theorem \ref{Main-1}.  Note that the axioms of He-Rapoport is verified in \cite[\S 7]{HR17} and the axiom \ref{uniform} is verified using Rapoport-Zink uniformization \cite[Theorem 6.1]{RZ96}. Therefore Corollary \ref{EKOR-DL} takes the following form 
\begin{equation*}
I_{\tau}(\QQ)\backslash X_{I, w}(\tau)\times G(\mathbb{A}^{p}_{f})/K^{p}\xrightarrow{\cong} \mathcal{A}_{I, w}
\end{equation*}
where $I_{\tau}$ is an inner form of $\GSp_{2g}$ and 
\begin{equation*}
X_{I, w}(\tau)=\bigsqcup_{i\in J_{\tau}/J_{\tau}\cap \breve{K}_{\supp_{\sigma}(w)}}iY_{\sigma}(w)
\end{equation*}
where $Y_{\sigma}(w)=\{g\breve{I}\in \breve{K}_{\supp_{\sigma}(w)}/\breve{I}: g^{-1}\tau\sigma(g)\in \breve{I}w\breve{I}\}.$

Let $J\subset I$ be a $\tau\sigma$-stable subset. We make the following manipulation, since $\supp_{\sigma}(w)\subset J$
\begin{equation*}
\begin{split}
X_{I, w}(\tau)&=\bigsqcup_{i\in J_{\tau}/J_{\tau}\cap \breve{K}_{\supp_{\sigma}(w)}}iY_{\sigma}(w)\\
&=\bigsqcup_{i\in J_{\tau}/J_{\tau}\cap \breve{K}_{J}}i \bigsqcup_{j\in J_{\tau}\cap \breve{K}_{J}/J_{\tau}\cap \breve{K}_{\supp_{\sigma}(w)}}jY_{\sigma}(w)\\
&=\bigsqcup_{i\in J_{\tau}/J_{\tau}\cap \breve{K}_{J}}i Y(J, w)\\
\end{split}
\end{equation*}
where $Y(J, w)=\{g\breve{I}\in \breve{K}_{J}/\breve{I}: g^{-1}\tau\sigma(g)\in \breve{I}w\breve{I}\}$. Consider the stratum $\mathcal{A}_{J, \tau}$ on $\mathcal{A}_{J}$. This is a finite set of points known as the \emph{minimal EKOR stratum}. In fact we have the following lemma.
\begin{lemma}\label{min}
$\mathcal{A}_{J, \tau}=I_{\tau}(\QQ)\backslash (J_{\tau}/J_{\tau}\cap \breve{K}_{J})\times G(\mathbb{A}^{p}_{f})/K^{p}.
$
\end{lemma}
\begin{proof}
By Corollary \ref{EKOR-DL}, we have $$I_{\tau}(\QQ)\backslash X_{K_{J}, \tau}(\tau)\times G(\mathbb{A}^{p}_{f})/K^{p}\xrightarrow{\cong} A_{J, \tau}.$$ Then we need to show that $(J_{\tau}/J_{\tau}\cap \breve{K}_{J})\cong X_{K_{J}, \tau}(\tau)$. In the case $\breve{K}_{J}=\breve{I}$, this is shown in \cite[Theorem 2.1.1]{GH15}. The general case can be deduced from this by using the fact that $\mathcal{A}_{I, \tau}\rightarrow \mathcal{A}_{J, \tau}$ is finite \'{e}tale as shown in  \cite{SYZ19}. See also \cite[Proposition 6.1]{GY10}.
\end{proof}

\begin{theorem}\label{GY-comp}
The basic KR stratum $\mathcal{A}_{I, w}$ is isomorphic to $\mathcal{A}_{J, \tau}\times Y(J, w)$.
\end{theorem}
\begin{proof}
This follows by summarizing the above discussions.
\end{proof}

Comparing this theorem with \cite[Corollary 6.5]{GY10} shows that the description of the basic KR stratum in Theorem \ref{Main-1} agrees with the description by G\"{o}rtz and Yu in \cite[Corollary 6.5]{GY10}. Note that contravariant  Dieudonn\'{e} module is used in \cite{GY10} and we are using covariant  Dieudonn\'{e} module here. This explains why we have $Y_{J}(w)$ instead of $X(w^{-1})$ in \cite[Corollary 6.5]{GY10}. 

\subsection{Basic EO stratum} Consider the isomorphism class of the $p$-torsion subgroup $A[p]$
together with its additional structures coming from a point $(A, \eta)$ on $\mathcal{A}_{\{0\}}=\mathcal{A}_{g}$.
After passing to its Dieudonn\'{e} module, we have a map
\begin{equation}\label{nuK}
\nu_{K}: \mathcal{A}_{g}\rightarrow G(\breve{\QQ}_{p})/\breve{K}_{\sigma}(\breve{K}_{1}\times \breve{K}_{1}).
\end{equation}
Let $^{g}W$ be the minimal length representatives of $S_{g}\backslash W_{g}$. As explained in Example \ref{Siegel}, the image of the map \eqref{nuK} can be identified with $^{g}W\subset W_{g}$. Let $c\leq g$, the finite Weyl group $W_{c}$ of $\Sp_{2c}$ can be viewed as a subgroup of $W_{g}$ naturally and we define ${^{c}W}$ exactly as in the case for ${^{g}W}$. Explicitly, the reflection $s_{c+1-i}$ in $W_{c}$ will be mapped to $s_{g+1-i}$ as a reflection in $W_{g}$. 
We have the following criterion for identifying the basic EO stratum. Let $K=\tilde{\mathbb{S}}-\{s_{0}\}$ which corresponds to the hyperspecial subgroup of $G(\QQ_{p})$. 

\begin{proposition}
Let $w\in ^{g}W$. Then the EO stratum $\mathcal{A}_{g, w}$ is basic if and only if $w\in {^{c}W}$ for some $c\leq g/2$. 
\end{proposition}
\begin{proof}
By Proposition \ref{basic-fin}, $EO_{K,w}$ is basic if and only if $w\in W_{\{c, g-c\}}\tau$ considered as an element in $^{K}\tilde{W}$. Since 
\begin{equation*}
^{c}W=W_{\{c, g-c\}}\cap {^{g}W}
\end{equation*} 
by \cite[Lemma 3.8]{GH12}, $EO_{K, w}$ is basic if and only if $w\in {^{c}W}$ as an element in ${^{g}W}$.
\end{proof}

This proposition rediscovers \cite[Remark 2.5.7]{Har10} and we now explain how to describe those basic EO strata as disjoint unions of Deligne-Lusztig varieties. Again Corollary \ref{EKOR-DL} takes the following form 
\begin{equation*}
I_{\tau}(\QQ)\backslash X_{K, w}(\tau)\times G(\mathbb{A}^{p}_{f})/K^{p}\xrightarrow{\cong} \mathcal{A}_{g, w}
\end{equation*}
and  by the proof of Theorem \ref{Main-1}
\begin{equation*}
X_{K,w}(\tau)=\bigsqcup_{i\in J_{\tau}/J_{\tau}\cap \breve{K}_{\supp_{\sigma}(w)\cup I(K, w, \sigma)}}iY_{I(K, w, \sigma)}(w)
\end{equation*}
where 
\begin{equation*}
\begin{split}
Y_{I(K, w, \sigma)}(w)&=\{g\breve{K}_{I(K, w, \sigma)}\in \breve{K}_{\supp_{\sigma}(w)\cup I(K, w, \sigma)}/\breve{K}_{I(K, w, \sigma)}: g^{-1}\tau\sigma(g)\in \breve{K}_{I(K, w, \sigma)}w\breve{K}_{\sigma(I(K, w, \sigma))}\}\\
&=\{g \breve{K}_{\supp_{\sigma}(w)\cap K} \in \breve{K}_{\supp_{\sigma}(w)}/\breve{K}_{\supp_{\sigma}(w)\cap K}: g^{-1}\tau\sigma(g)\in {\breve{K}_{\supp_{\sigma}(w)\cap K \sigma}}\breve{I}w\breve{I} \}.\\
\end{split}
\end{equation*}

\begin{lemma}\label{5.5}
Suppose that $w\in {^{c}W}$ and $w\not\in {^{c-1}W}$, then 
\begin{equation*}
\supp_{\sigma}(w)=\{s_{0}, \cdots, s_{c-1}\}\cup \{s_{g-c+1}, \cdots, s_{g}\}.
\end{equation*}
\end{lemma}

\begin{proof}
By the proof of \cite[Lemma 7.1]{Hoe10}, see also \cite[Lemma 3.8]{GH12}, the set of simple reflections in a reduced expression of $w$ is equal to $\{s_{i}, s_{i+1}, \cdots, s_{g-1}, s_{g}\}$ with $i>g-c$. Since $w\not\in {^{c-1}W}$, $i=g-c+1$. Then by definition we have 
\begin{equation*}
\supp_{\sigma}(w)=\{s_{0}, \cdots, s_{c-1}\}\cup \{s_{g-c+1}, \cdots, s_{g}\}.
\end{equation*}
\end{proof}

Note the proof of the above lemma shows that those $w\in {^{c}W}$ and $w\not\in {^{c-1}W}$ are in fact $\sigma$-Coxeter elements in $W_{\supp_{\sigma}(w)}$. Therefore we obtain the following lemma.

\begin{lemma}\label{5.6}
The set $I(K, w, \sigma)$ is equal to $\{s_{c+1}, \cdots, s_{g-c-1}\}$.
\end{lemma}
\begin{proof}
By \cite[Lemma 4.6.1]{GH15}, $I(K, w, \sigma)$ is exactly the set of simple reflections that commute with \begin{equation*}
\supp_{\sigma}(w)=\{s_{0}, \cdots, s_{c-1}\}\cup \{s_{g-c+1}, \cdots, s_{g}\}.
\end{equation*}
The result is then clear.
\end{proof}

For $I(K, w, \sigma)=\{s_{c+1}, \cdots, s_{g-c-1}\}$, we denote by $Y_{c}(w)$ the fine Deligne-Lusztig variety given by
\begin{equation*}
Y_{I(K, w, \sigma)}(w)=\{g \breve{K}_{\supp_{\sigma}(w)\cap K} \in \breve{K}_{\supp_{\sigma}(w)}/\breve{K}_{\supp_{\sigma}(w)\cap K}: g^{-1}\tau\sigma(g)\in {\breve{K}_{\supp_{\sigma}(w)\cap K \sigma}}\breve{I}w\breve{I} \}.
\end{equation*}
Since $K=\tilde{S}-\{s_{0}\}$, $K\cap \supp_{\sigma}(w)=\{s_{1}, \cdots, s_{c-1}\}\cup\{s_{g-c+1}, \cdots, s_{g}\}$. It follows that $Y_{c}(w)$ is exactly the fine Deligne-variety denoted by $X_{c(w)}\{\tau(w)\}$ in \cite[Theorem 1.2]{Hoe10}.

\begin{theorem}\label{Ho-comp}
For  $w\in {^{c}W}$ and $w\not\in {^{c-1}W}$, the basic EO stratum $\mathcal{A}_{g, w}$ is isomorphic to $\mathcal{A}_{\{c, g-c\}, \tau}\times Y_{c}(w)$
\end{theorem}
\begin{proof}
We know
\begin{equation*}
I_{\tau}(\QQ)\backslash X_{K, w}(\tau)\times G(\mathbb{A}^{p}_{f})/K^{p}\xrightarrow{\cong} \mathcal{A}_{g, w}
\end{equation*}
and 
\begin{equation*}
X_{K,w}(\tau)=\bigsqcup_{i\in J_{\tau}/J_{\tau}\cap \breve{K}_{\supp_{\sigma}(w)\cup I(K, w, \sigma)}}iY_{c}(w).
\end{equation*}
Notice that $\supp_{\sigma}(w)\cup I(K, w, \sigma)=\tilde{\mathbb{S}}-\{c, g-c\}$ by Lemma \ref{5.5} and Lemma \ref{5.6}. The result follows from the fact that 
\begin{equation*}
I_{\tau}(\QQ)\backslash J_{\tau}/J_{\tau}\cap \breve{K}_{\tilde{\mathbb{S}}-\{c, g-c\}}\times G(\mathbb{A}^{p}_{f})/K^{p}.
\end{equation*}
is isomorphic to $\mathcal{A}_{\{c, g-c\}, \tau}$ as shown in Lemma \ref{min}.
\end{proof}

This shows that the descriptions of G\"{o}rtz-He-Nie in Theorem \ref{Main-1} for the basic EO stratum agrees with 
the description of Hoeve as in \cite[Theorem 1.2]{Hoe10}. Moreover let $J=\tilde{\mathbb{S}}-\{c, g-c\}$ in Theorem \ref{GY-comp} and let $w\in {^{c}W}$ as in Theorem \ref{Ho-comp}. Then we easily  obtain the following commutative diagram combining  Theorem \ref{GY-comp} and Theorem \ref{Ho-comp} 
\begin{equation*}
\begin{tikzcd}
 \mathcal{A}_{I, w} \arrow[r,  "\sim"] \arrow[d]& \mathcal{A}_{\{c, g-c\}, \tau}\times Y(J, w)  \arrow[d]\\
   \mathcal{A}_{g, w} \arrow[r,  "\sim"]  & \mathcal{A}_{\{c, g-c\}, \tau}\times Y_{c}(w)\\
\end{tikzcd}
\end{equation*}
with the vertical arrows the natural projection maps. This recovers one of the main theorems of \cite[Theorem 1.1(3)]{GH12}.

 \end{document}